\documentclass[12pt]{amsart}

\usepackage[utf8]{inputenc}
\usepackage{hyperref}
\usepackage{lmodern}
\usepackage{comment}

\usepackage{amssymb,amsthm,amsmath,mathrsfs,amssymb}

\relax

\usepackage[a4paper,margin=3cm]{geometry}
\usepackage[mathscr]{euscript}

\makeatletter
\def\paragraph{\@startsection{paragraph}{4}%
  \z@\z@{-\fontdimen2\font}%
  {\normalfont\bfseries}}
\makeatother

\newcommand{\eps}{\varepsilon}

\newcommand{\Z}{\mathbb{Z}}
\newcommand{\N}{\Z_{\geq 0}}

\newcommand{\R}{\mathbb{R}}
\newcommand{\C}{\mathbb{C}}

\newcommand*\colvec[3][]{
    \begin{pmatrix}\ifx\relax#1\relax\else#1\\\fi#2\\#3\end{pmatrix}
}
\newcommand*\abs[1]{
    \left|#1\right|
}
\newcommand*\ad[1]{
    \mathrm{ad}\,#1
}
\newcommand*\pscal[2]{
	\langle#1,#2\rangle
}

\newcommand*\Leb{
	\mathrm{L}
}
\newcommand*\clsC{
	\mathscr{C}
}
\newcommand*\Sobolev{
	\mathrm{H}
}
\newcommand*\Poids{
	\mathscr{P}
}

\newcommand{\norm}[1]{\left\lVert#1\right\rVert}
\newtheorem{thm}{Theoreme}
\newtheorem{lem}{Lemma}

\newtheorem{prop}{Proposition}

\theoremstyle{definition}

\newtheorem{rem}{Remark}

\newtheorem{fait}{Claim}
\title[Absolutely continuous Furstenberg measures]{Absolutely continuous Furstenberg measures for finitely-supported random walks}
\author{Félix Lequen}
\address{Laboratoire AGM -- CY Cergy Paris Université}
\email{felix.lequen@cyu.fr}
\begin{document}

\maketitle

	\begin{abstract}In this note, we generalise a Bourgain's construction of finitely-supported symmetric measures whose Furstenberg measure has a smooth density from the case of $\mathrm{SL}_2(\R)$ to that of a general simple Lie group. The proof is the same as Bourgain's, except that the use of Fourier series is replaced by harmonic analysis on a maximal compact subgroup.\end{abstract}

\section{Introduction}
Let $\mu$ be a Borel probability measure on a non-compact connected simple Lie group $G$ with finite centre. Assuming that the support of $\mu$ generates a Zariski-dense semi-group, by a theorem of Furstenberg and Goldshei'd-Margulis, there exists a unique measure $\nu$ which is stationary for the random walk induced by $\mu$ on the full flag manifold $\mathscr{F}$ of $G$, that is $\mu * \nu = \nu$, where $\mu * \nu$ is the image of $\mu \otimes \nu$ by the action map $(g,\xi) \mapsto g\xi$. The measure $\nu$ is called the \emph{Furstenberg measure} and it controls many properties of the random walk associated to $\mu$. The study of the geometric properties of $\nu$ is therefore of interest (see Benoist-Quint \cite{benoist2018regularity} for more background).

In this note we focus on the case of finitely-supported $\mu$. Kaimanovich-Le Prince \cite{kaimanovich-leprince} had conjectured that in this case, the measure $\nu$ is never absolutely continuous with respect to the Lebesgue measure on $\mathscr{F}$ and in fact constructed examples with arbitrary small Hausdorff dimension of $\nu$:

\begin{thm}
	Let $d \geq 2$ and $\Gamma$ be a finitely-generated Zariski-dense subgroup of $\mathrm{SL}_d(\R)$. Then for every $\delta > 0$, there exists a finitely-supported measure $\mu$ on $\Gamma$ such that the Furstenberg measure associated to $\mu$ has Hausdorff dimension less than $\delta$.
\end{thm}

 However, Barany-Pollicott-Simon \cite{bps} showed this conjecture to be false using the "transversality method" (Pollicott-Simon \cite{pollicott-simon}), initially invented in the context of self-similar sets and which had given striking results for the study of Bernoulli convolutions (see Solomyak \cite{solomyak} and Peres-Solomyak \cite{peres-solomyak}), with which there is a number of analogies to the problem studied here.

The construction of Barany-Pollicott-Simon was not explicit, however, and did not give examples of symmetric measures. Another construction was given by Bourgain \cite{bourgain2012finitely} in the case of $\mathrm{SL}_2(\R)$, which gives rather explicit (symmetric) examples, and then Benoist-Quint \cite{benoist2018regularity} in the more general case of a connected semi-simple Lie group which is different from $\mathrm{SL}_2(\R)$. Actually, both proof give more than absolute continuity, as they show that it is possible to have a density of class $\mathscr{C}^r$ for arbitrary large $r$.

A fundamental part of Bourgain's proof (and in a different way of Benoist-Quint's proof) is the use of finitely-supported measure having a \emph{restricted spectral gap} for the associated Markov operator in $\mathrm{SL}_2(\R)$, which had been constructed earlier, first in the simpler compact case by the Bourgain-Gamburd method \cite{bourgain-gamburd-su2}\cite{bourgain-gamburd-sud}\cite{benoist-saxce}, and then in the non-compact case for $\mathrm{SL}_2(\R)$ by Bourgain-Yehudahoff \cite{bourgain-yehudayoff}. Since Bourgain's proof appeared, the restricted spectral gap property was established beyond the $\mathrm{SL}_2(\R)$ case by Boutonnet-Ioana-Salehi Golsefidy \cite{bisg}. More precisely, their proof gives the following (see the proof of corollary C in \cite{bisg}):

\begin{thm}
	Let $\Gamma < G$ be a \emph{topologically dense} subgroup. Assume that there exists a basis of $\mathfrak{g}$ with respect to which for every $g\in \Gamma$, the matrix of $\mathrm{Ad}\,g$ has coefficients which are algebraic numbers.
	
	Let $U$ be a neighbourhood of the identity in $G$. Then there exists a symmetric subset $T \subset U \cap \Gamma$ such that the Markov operator $T$ associated to the measure $\mu = \frac{1}{\abs{T}}\sum_{g \in T} \delta_g$ on $\Leb^2(\mathcal{F})$ defined by
		\[(Tu)(\xi) := \int u\left(g\xi\right)\mathrm{d}\mu(g)\] for every $u \in \Leb^2(\mathcal{F})$ and $\xi \in \mathscr{F}$ has a \emph{restricted sepctral gap}, that is, there exists a subspace $V$ of finite dimension in $\Leb^2(\mathscr{F})$such that for every $u \in V^{\perp}$, $\norm{Tu}_2 \leq \frac{1}{2}\norm{u}_2$.
		\label{thm:bisg}
\end{thm}

Let us now note that a functional-analytic argument of Benoist-Quint \cite{benoist2018regularity} already gives a density in $\Leb^2(\mathscr{F})$ from this last statement: indeed the restricted spectral gap of $T$ implies that the essential spectral radius of $T$ is strictly less than $1$ (in fact less than $1/2$), which implies that $1 = \dim \ker\left(T - 1\right) = \dim \ker\left(T^* - 1\right)$, where $T^{*}$ is the adjoint for the $\Leb^2$ scalar product. But a stationary measure with density in $\Leb^2$ is the same as an eigenvector of $T^{*}$ for the eigenvalue $1$. So the interest of what follows is only to obtain higher regularity. To simplify the proof below, we will make use of the existence of the density in $\Leb^2$.

In this note we focus on generalising Bourgain's argument which deduces from theorem \ref{thm:bisg} properties of regularity for the stationary measure associated to $\nu$. Bourgain's argument for the action of $\mathrm{SL}_2(\R)$ acting on $\mathbb{P}^1$ relies on the use of Fourier series on $\mathbb{P}^1$ seen as a quotient of the circle. For the general case, we note that a maximal compact subgroup $K$ of $G$ acts transitively on $\mathscr{F}$ and therefore we can similarly use the representation theory of $K$. The argument is then exactly parallel to that of Bourgain. More precisely, below we show the following: 

\begin{prop}Let $t > 0$ be a real number. Then there exists a neighbourhood $U$ of the identity in $G$ such that the following holds:
	
	Let $\mu$ be a measure with finite support $S$ and which generates a dense subgroup of $G$ such that $S \subset U$. Assume that the operator $T \colon \Leb^2(\mathcal{F}) \to \Leb^2(\mathcal{F})$ defined by
	\[\left(Tu\right)(\xi) := \sum_{g \in S}\mu(g)u(g^{-1}\xi)\]
	for all $u \in \Leb^2(\mathscr{F})$,
	satisfies the following condition: there exists a subspace $V$ of finite dimension in $\Leb^2(\mathscr{F})$ such that for every $u \in V^{\perp}$, $\norm{Tu}_2 \leq \frac{1}{2}\norm{u}_2$.
	
	Then the unique $\mu$-stationary measure $\nu$ is absolutely continuous with density in $\Sobolev^t(\mathcal{F})$.\label{lem}
\end{prop}

Combining theorem \ref{thm:bisg} and proposition \ref{lem}, we obtain the following result:

\begin{thm}
		Let $t > 0$ be a real number. 
		Let $\Gamma < G$ be a \emph{topologically dense} subgroup. Assume that there exists a basis of $\mathfrak{g}$ with respect to which for every $g\in \Gamma$, the matrix of $\mathrm{Ad}\,g$ has coefficients which are algebraic numbers.
		
		Then there exists a symmetric subset $T \subset \Gamma$ such that the random walk associated to the measure $\frac{1}{\abs{T}}\sum_{g \in T}\delta_g$ has a density in $\Sobolev^t(\mathscr{F})$ (and therefore in $\mathscr{C}^k$ if $t$ is large enough by the Sobolev embedding theorem).
\end{thm}

Let us note that the higher the regularity, the closer to the identity the measure is required to be supported. This is somewhat similar in spirit to a result of  Erd\H{o}s-Kahane \cite{sixty}. 

It remains an interesting problem to study the regularity of arbitrary finitely-supported measures whose support is close enough to the identity. In this direction, let us mention the results of Hochman-Solomyak \cite{hochman-solomyak} on the dimension of the Furstenberg measure .

\paragraph{Acknowledgement. } I thank Nicolas de Saxcé for having read an earlier version of this note, for interesting discussions on this topic and for letting me know of his slightly different proof of a similar result with Wouter van Limbeek and David Fisher. I also thank Bertrand Deroin for encouragement and advice on this topic.
\section{Preliminaries}
\paragraph{Notations. }
Let $G$ be a connected non-compact simple Lie group with finite centre. Let $\mathfrak{g}$ be its Lie algebra. Let $\theta$ be a Cartan involution and write the associated Cartan decomposition $\mathfrak{g} = \mathfrak{k} \oplus \mathfrak{p}$ with $\mathfrak{k}$ and $\mathfrak{p}$ the eigenspaces of $\theta$ associated to the eigenvalues $1$ and $-1$, respectively.

Let $B \colon (X, Y) \mapsto \mathrm{tr}(\ad X\ad Y)$ be the Killing form of $\mathfrak{g}$. We consider the scalar product $(X, Y) \mapsto \pscal{X}{Y} := -B(X, \theta Y)$ on $\mathfrak{g}$, which is $\mathrm{Ad}$-invariant, and the associated norm denoted by $\norm{\cdot}$.

Let $\mathfrak{a} \subset \mathfrak{p}$ be a maximal abelian subspace and $A = \exp \mathfrak{a}$. Denote by $\Sigma$ the set of restricted roots. We choose a closed Weyl chamber $\mathfrak{a}^{+}$ of $\mathfrak{a}$ and denote by $\Sigma^{+}$ the set of associated positive roots, that is those which are non-negative on $\mathfrak{a}^{+}$. Let $\mathfrak{n} := \sum_{\lambda \in \Sigma^{+}}\mathfrak{g}_\lambda$ where $\mathfrak{g}_\lambda := \{X\in\mathfrak{g}\,:\,\left(\mathrm{ad}\,H\right)X = \lambda(H)X\text{ pour tout }H \in \mathfrak{a}\}$ is the root space associated to $\lambda \in \mathfrak{a}^{*}$ and $N$ the analytic subgroup with Lie algebra $\mathfrak{n}$. Let us also write \[\rho := \frac{1}{2}\sum_{\lambda \in \Sigma^{+}} \lambda\in \mathfrak{a}^{*}\] for the half-sum of positive roots.

Define $A^{+} := \exp \mathfrak{a}^{+}$. Then we have the Cartan decomposition $G = KA^{+}K$: for every $g \in G$, there exists a unique $\kappa(g) \in \mathfrak{a}^{+}$ such that $g \in K\exp\left(\kappa(g)\right)K$. This defines a map  $\kappa \colon G \to \mathfrak{a}$ called the \emph{Cartan projection}. Similarly, the Iwasawa decomposition $G = KAN$ gives the existence, for every $g \in G$, of a unique $H(g) \in \mathfrak{a}$ such that $g \in K\exp\left(H(g)\right)N$.

Let $M$ be the centraliser of $A$ in $K$. Let $P = MAN$ be the associated standard minimal parabolic subgroup, and $\mathcal{F} = G/P$ the associated flag manifold. The group $K$ acts transitively on $\mathcal{F}$ and the stabiliser is $M$. We will therefore identify $\mathcal{F}$ with $K/M$. The scalar product on $\mathfrak{g}$ induces a Riemannian metric on $K/M$ which is $K$-invariant on the left. Let $m$ be the associated volume measure, which is a Haar measure.

Given $g\in G$ and $\xi = kP \in \mathcal{F}$ with $k \in K$, write $\sigma(g, \xi) = H(gk^{-1})$; this defines a cocycle $\sigma \colon G \times \mathcal{F} \to \mathfrak{a}$ called the \emph{Iwasawa} cocycle. We then have the following lemma \cite{quint2006overview}:
\begin{lem}The Radon-Nikodym derivative of $(g^{-1})_{*}m$ with respect to $m$ at $\xi \in \mathcal{F}$ is $e^{-2\rho(\sigma(g,\,\xi))}$.\end{lem}
Moreover, there is an inequality between the Iwasawa cocyle and the Cartan projection \cite[consequence of corollary 8.20]{benoistrandom}:
\begin{lem}For every $g \in G$ and every $\xi \in \mathcal{F}$, we have $\norm{\sigma(g, \xi)} \leq \norm{\kappa(g)}$.
	\label{lem:cartan-iwasawa}\end{lem}

In the following, given two quantities $A$ and $B$, we write $A \lesssim B$ if there exists a constant $C > 0$ depending only on the group $G$ and the choices made above (the group $K$, the Cartan involution, etc.) such that $A \leq CB$, and $A \asymp B$ if $A \lesssim B$ and $B \lesssim A$. If the implied constants depend additionally on other parameters, we will write them in indices, for instance $A \lesssim_s B$ if the implied constant depends on $s$.

\paragraph{Sobolev spaces on a representation. }In the following, we will need some facts on Sobolev spaces on Lie group representations, as in \cite{edwards2017rate} for instance.

 Consider the representation $\pi \colon G  \to \mathrm{U}(\Leb^2(\mathcal{F}))$ defined by
 \[\left(\pi(g)u\right)(\xi) := u(g^{-1}\xi)e^{-\rho(\sigma(g, \xi))}\] 
 for every $u \in \mathscr{C}^\infty(\mathcal{F})$ and $\xi \in \mathcal{F}$. This defines an irreducible unitary representation \cite[chapter VII]{knapp2016representation}. Here we use the standard Hilbert space structure on $\Leb^2(\mathcal{F})$ and denote hy $\pscal{\cdot}{\cdot}$ the inner product and $\norm{\cdot}$ the norm. 
 
 By differentiation, we define a representation of the Lie algebra: for every $X \in \mathfrak{g}$ and $v \in \mathscr{C}^\infty(\mathcal{F})$
\[\pi(X)v := \frac{\mathrm{d}}{\mathrm{d}t}\Bigr|_{\substack{t = 0}}\pi(e^{tX})v\] which we can extend to the universal envelopping algebra $\mathcal{U}(\mathfrak{g}_\C)$  $\mathfrak{g}_\C$, and that we also denote by $\pi$ (see Knapp \cite[chapter III]{knapp2016representation}).

We now define Sobolev norms. We will need several equivalent definitions. For definitess we fix one: let ${X_i}$ be an orthonormal basis of $\mathfrak{k}$ for the inner product on $\mathfrak{g}$, and similarly let ${Y_i}$ be an orthonormal basis of $\mathfrak{p}$. Let
\[\Delta = -\sum_i X_i^2 -\sum_i Y_i^2 \in \mathcal{U}(\mathfrak{g}_\C),\]
and consider, for any $s \in \N$, the the scalar product $\pscal{\cdot}{\cdot}_{\Sobolev^s}$ on $\mathscr{C}^\infty(\mathcal{F})$ defined by
\[\pscal{u}{v}_{\Sobolev^s} := \pscal{\pi(1 + \Delta)^su}{v}.\]
We will write $\norm{\cdot}_{\Sobolev^s}$ for the associated norm. Then we define the space $\Sobolev^s(\mathscr{F})$ as the closure in $\Leb^2(\mathscr{F})$ of $\mathscr{C}^\infty(\mathscr{F})$ for the norm $\norm{\cdot}_{\Sobolev^s}$.

\paragraph{Harmonic analysis on $\Leb^2(\mathcal{F})$. }As we have seen, we can identify $\mathcal{F}$ with $K/M$. Actually, denoting by $K_0$ the neutral component of $K$, the compact connected Lie group $K_0$ acts transitively on $\mathcal{F}$, too, with stabiliser $K_0 \cap M$ (see Knapp \cite[lemma 7.33]{knapplie}). Therefore we identify $\Leb^2(\mathcal{F})$ as a $K_0$-module to the subspace of right-$\left(K_0\cap M\right)$-invariant elements of $\Leb^2(K_0)$.

Let $\hat{K_0}$ be the unitary dual of $K_0$.
The differential operator $\Delta$ can be written as $\Delta = -\mathcal{C} - 2\mathcal{C}_K$, where $\mathcal{C}$ and $\mathcal{C}_K$ are the Casimir operators of $G$ and $K$ respectively (see Knapp \cite[proof of theorem 8.7]{knapp2016representation}). Because $\pi$ is an irreducible unitary representation, $\pi(\mathcal{C})$ acts as a constant on $\mathscr{C}^\infty(\mathcal{F})$. As the Casimir operator of $K$ acts as an elliptic operator on $\mathscr{F}$, this means that this Sobolev norm coincides with the usual definition of a Sobolev norm on the Riemannian manifold $\mathscr{F}$.

Moreover, for every $\tau \in \hat{K_0}$, the operator $1 + \Delta$ acts on the subspace $\Leb^2(\mathscr{F})_\tau$ of $\tau$-isotypic vectors as a constant $c(\tau)$. Because $ \pi\left(\Delta\right)$ is self-adjoint and $\pscal{\pi\left(\Delta\right) u}{u} \geq 0$ for every $u \in \mathscr{C}^\infty(\mathscr{F})$, this constant is a real number and $c(\tau) \geq 1$. 

We now define a Littlewood-Paley decomposition. For every non-negative integer $k$, let $\mathcal{L}_k$ be the orthogonal sum of the $\Leb^2(\mathscr{F})_\tau$, for all $\tau \in \hat{K_0}$ such that $2^k \leq c(\tau) < 2^{k + 1}$. Then $\Leb^2(\mathscr{F})$ is the Hilbert sum of the $\mathcal{L}_k$ ($k \in \N$). For every $k\in\N$, let $P_k$ be the orthogonal projection on $\mathcal{L}_k$. Write also $P_{< k} = \sum_{0 \leq j < k} P_j$ and $P_{\geq k} = 1 - P_{< k}$. Then we can give a second, equivalent, definition of the Sobolev norm. Let $s \in \N$. For $u\in \Leb^2(\mathscr{F})$, then
\[\norm{u}_{\Sobolev^s}^2 \asymp_s \sum_{k} 2^{sk}\norm{P_ku}_2^2 .\]

Finally, in the course of the proof of lemma \ref{lem:sob-repr}, we will need a third definition of Sobolev norms. Let $\mathcal{B}$ be the basis given by the $\{X_i\}$ and the $\{Y_i\}$ as above in the definition of $\Delta$. Given a non-negative integer $s$ and $u \in \mathscr{C}^\infty(\mathcal{F})$, we have (see Nelson \cite{nelson1959analytic})
\begin{equation}\norm{u}_{\Sobolev^s}^2 \asymp_{s} \sum_{k = 0}^s \sum_{X_1, \ldots, X_k \in \mathcal{B}} \norm{\mathrm{d}\pi(X_{1}\ldots X_{k})u}_2^2
\label{eq-sobolev-3}\end{equation}
where for $k = 0$ the sum is reduced to $\norm{u}_2^2$. 
\begin{lem}
	There exists a constant $c > 0$ such that for every $s \in \N$, and for every $g \in G$ and $u\in \mathscr{C}^\infty(\mathscr{F})$, $\norm{\pi(g)u}_{\Sobolev^s} \lesssim_s e^{cs\norm{\kappa(g)}}\norm{u}_{\Sobolev^s}$.
	\label{lem:sob-repr}
\end{lem}
\begin{proof}
	For $s = 0$, by lemma \ref{lem:cartan-iwasawa}, we have
	$\norm{\pi(g)u}_2 \leq e^{\norm{\rho}\norm{\kappa(g)}}\norm{u}_2$. For $s = 1$ and $X \in \mathfrak{g}$, we have
	\begin{align*}\norm{\mathrm{d}\pi(X)\pi(g)u}_2 &= \norm{\pi(g)\mathrm{d}\pi((\mathrm{Ad}\,g^{-1})X)u}_2 \\
	&\leq \norm{\pi(g)}_2\norm{\mathrm{d}\pi((\mathrm{Ad}\,g^{-1})X)u}_2 \\ 
	&\lesssim_s e^{c\norm{\kappa(g)}}\norm{u}_{\Sobolev^1}\end{align*} for some constant $c > 0$, which implies the result by equation \ref{eq-sobolev-3}. Here we used the Cartan decomposition to bound the coefficients of $\left(\mathrm{Ad}\,g^{-1}\right)X$ in the basis $\mathcal{B}$.
	The general result follows similarly by induction.
\end{proof}

\section{Bourgain's argument}
\paragraph{Decay of Fourier coefficients. }
We now state a slightly more precise form of proposition \ref{lem} and prove it:
\begin{prop}There exists a constant $C > 0$ such that the following holds:
	
	Let $\eps > 0$ be small enough. Let $\mu$ be a measure whose support has finite support $S$ and generates a dense subgroup of $G$ such that for every $g \in S$, $\norm{\kappa(g)} \leq \eps$. Assume that the operator $T \colon \Leb^2(\mathcal{F}) \to \Leb^2(\mathcal{F})$ defined by
	\[\left(Tu\right)(\xi) := \sum_{g \in S}\mu(g)u(g^{-1}\xi)\]
	for all $u \in \Leb^2(\mathscr{F})$,
	satisfies the following condition: there exists a subspace $V$ of finite dimension in $\Leb^2(\mathscr{F})$ such that for every $u \in V^{\perp}$, $\norm{Tu}_2 \leq \frac{1}{2}\norm{u}_2$.
	
	Then the unique $\mu$-stationary measure $\nu$ is absolutely continuous with density in $\Sobolev^t(\mathcal{F})$ for every $t < \frac{C}{\eps}$.\label{lem}
\end{prop}
	\begin{proof}[Proof of proposition \ref{lem}]As we have noted in the introduction, the arguments of Benoist-Quint \cite{benoist2018regularity} show that there exists a density in $\Leb^2(\mathscr{F})$ for the stationary measure $\nu$; let us denote it by $g$.
		
		We now reduce to the case where there exists $N$ such that the subspace $V$ is the sum of the $\mathcal{L}_k$ for $k < N$. Indeed, we can approximate a finite-dimensional subspace in such a way; this might slightly increase the norm of the operator $T$, but up to replacing it by $T^2$ this is not a problem, as long as $\eps$ is small enough.
		\begin{fait} There exists a constant $c > 0$ such that for every non-negative integer $s$ which is large enough and every $m \in \N$, we have
			 $\norm{\left(T^{*}\right)^m}_{\Sobolev^s} \lesssim_s e^{csm\eps}$.
		\end{fait}
		Here we have denote by $T^{*}$ the $\Leb^2$-adjoint of $T$.\begin{proof}
	We can reduce to the analogous statement where $\mu$ is a Dirac mass at $g \in G$, in which case for every $u \in \Leb^2(\mathscr{F})$, $Tu = \pi(g)u \cdot \pi(g)1$. The lemma is then a consequence of the bound for the Sobolev norm of a product and lemma \ref{lem:sob-repr}.
		\end{proof}
We now give a bound for the low frequencies. Let $\tau, \sigma \in \hat{K_0}$ and $u \in \Leb(\mathscr{F})_\tau$ et $v \in \Leb(\mathscr{F})_\sigma$. Then for every even non-negative integer $s$, we have:
\begin{align*}
\abs{\pscal{T^mu}{v}} &=\abs{\pscal{u}{\left(T^{*}\right)^mv}} \\
&= \frac{1}{c(\tau)^{s/2}}\abs{\pscal{\pi\left(1 + \Delta\right)^{s/2}u}{\left(T^{*}\right)^mv}} \\
&\lesssim_s \frac{1}{c(\tau)^{s/2}}\norm{u}_2\norm{\left(T^{*}\right)^mv}_{\Sobolev^s} \\
&\leq \frac{1}{c(\tau)^{s/2}}\norm{\left(T^{*}\right)^m}_{\Sobolev^s}\norm{u}_2\norm{v}_{\Sobolev^s} \\
&\asymp_s \frac{c(\sigma)^{s/2}}{c(\tau)^{s/2}}\norm{\left(T^{*}\right)^m}_{\Sobolev^s}\norm{u}_2\norm{v}_2.
\end{align*}
Let now $k \in \N$ and $u \in \mathcal{L}_k$. Write $u = \sum_{\tau} u_\tau$ where $\tau$ ranges over those $\tau \in \hat{K_0}$ with $2^k \leq c(\tau) < 2^{k + 1}$ and $u_\tau$ is the $\tau$-isotypic component of $u$. Let $N_k$ be the number of $\tau \in \hat{K_0}$ such that $2^{k} \leq c(\tau) < 2^{k + 1}$. A classical argument relating $c(\tau)$ to a quadratic expression in the highest weight of $\tau$ (see Warner \cite[proof of lemma 4.4.2.3]{warner2012harmonic}) shows that $N_k \asymp 2^{rk/2}$, where $r$ is that rank of $K_0$. Therefore
\begin{align*}
\abs{\pscal{T^mu}{v}} &\leq N_k^{1/2}\left(\sum_{2^k \leq c(\tau) < 2^{k + 1}} \pscal{T^mu_\tau}{v}^2\right)^{1/2} \\
&\leq \norm{\left(T^{*}\right)^m}_{H^s}N_k^{1/2}\left(\sum_{2^k \leq c(\tau) < 2^{k + 1}} \frac{c(\mu)^s}{c(\lambda)^s}\norm{u_\tau}^2\norm{v}^2\right)^{1/2} \\
&\leq 2^{rk/4}\norm{\left(T^{*}\right)^m}_{\Sobolev^s}c(\mu)^{s/2}2^{-sk}\norm{u}_2\norm{v}_2
\end{align*}

Let $j \in \N$. Taking an orthonormal basis of $\mathcal{L}_j$ gives:
\[\norm{P_jT^mu}_2 \lesssim_{s, j} \norm{\left(T^{*}\right)^m}_{\Sobolev^s}2^{(r/4 - s)k} \norm{u}_2.\]

We will need the following Bernstein-type inequality \cite{howe1988almost} for the sup-norm $\norm{\cdot}_\infty$, which is a simple consequence of the representation theory of $K_0$ and the Cauchy-Schwarz inequality:
\begin{lem}Let $\tau \in \hat{K_0}$ and $u \in \Leb^2(K_0)_\tau$. Then $\norm{u}_\infty \leq \dim(\tau)\norm{u}_2$, where $\dim(\tau)$ denotes the dimension of $\tau$.
\end{lem}

Therefore, for every $j \in \N$, we have \begin{equation}\norm{P_{\leq j}T^mu}_\infty \lesssim_{s, j}\norm{\left(T^{*}\right)^m}_{H^s}2^{(r/4 - s)k}\norm{u}_2.
\label{eq:low-freq}
\end{equation}
	
We now give Bourgain's bound, which relies on iterating the restricted spectral gap of $T$ for the high frequencies, and controlling the low frequencies part which appear at each inductive step with equation \ref{eq:low-freq}. Let $k \in \N$ and $u \in \mathcal{L}_k$. Then
\begin{align*}\norm{T^{m + 1} u}_2 
&\leq \norm{TP_{< N}T^mu}_2 + \norm{TP_{\geq N} T^mu}_2 \\
&\leq m(\mathscr{F})^{1/2}\norm{TP_{< N}T^mu}_{\infty} + \frac{1}{2}\norm{P_{\geq N}T^mu}_2 \\
&\leq m(\mathscr{F})^{1/2}\norm{P_{< N}T^mu}_\infty + \frac{1}{2}\norm{T^mu}_2\\
&\leq C_{s, N}\norm{\left(T^{*}\right)^m}_{H^s}2^{(r/4 - s)k}\norm{u}_2 + \frac{1}{2}\norm{T^mu}_2
\end{align*}
for some constant $C_{s, N} > 0$.
	
Iterating, we obtain:
\[\norm{T^\ell u}_2 \lesssim_{s, N} \left(\frac{\norm{\left(T^{*}\right)^{\ell - 1}}_{\Sobolev^s}}{2^{\ell - 1}} + \ldots + \frac{\norm{T^{*}}_{\Sobolev^s}}{2} \right)2^{(r/4 - s)k}\norm{u}_2 + 2^{-\ell}\norm{u}_2.\]
Because for $s$ large enough $\norm{T^m}_{\Sobolev^s} \leq 1 \leq C_se^{csm\eps}$ by the claim, we thus have:
\[\norm{T^\ell u}_2 \lesssim_{s, N} \left(e^{cs\ell\eps}2^{(r/4 - s)k} + 2^{-\ell}\right)\norm{u}_2.\]

Let $\ell$ be the integer part of
\[\frac{s - r/4}{\log 2 + cs\eps}k.\]
Then \[\norm{T^\ell u}_2 \lesssim_{s, N} \exp\left(-\left(s - r/4\right)\frac{\log 2}{\log 2 + cs\eps}k\right)\norm{u}_2.\]

For $s$ large enough $\left(s - r/4\right)\frac{\log 2}{\log 2 + cs\eps} > \frac{\log 2}{2c\eps}$, which implies that for $k$ large enough
\[\norm{T^\ell u}_2 \lesssim_{s, N} 2^{-\frac{k}{2c\eps}}\norm{u}_2,\] and therefore
\[\abs{\int u \mathrm{d}\nu} = \abs{\int T^\ell u\mathrm{d}\nu} = \pscal{T^\ell u}{g} \lesssim_{s, N} 2^{-\frac{k}{2c\eps}}\norm{u}_2.\]
where we recall that $g$ is the $\Leb^2$ density of $\nu$.
Therefore, for every $k \in \N$ lare enough,
\[\norm{P_kg}_2^2 = \pscal{P_kg}{g} \lesssim_{s, N}2^{-\frac{k}{2c\eps}}\norm{P_kg}_2.\]
This implies that $\norm{P_kg}_2 \lesssim_{s, N} 2^{-\frac{k}{2c\eps}}$ for $k$ large enough and therefore that $g \in \Sobolev^{t}(\mathscr{F})$ for every $t < \frac{1}{2c\eps}$.
\end{proof}

\bibliographystyle{plain}
\bibliography{bourgain}
\end{document}